\documentclass[10pt,leqno]{article}
\usepackage{verbatim}
\usepackage{amssymb}
\usepackage{mathtools}
\usepackage{manfnt}
\usepackage{url}%
\usepackage{amsrefs}
\usepackage{rotating}
\usepackage{amsmath}
\usepackage{tabularx}
\usepackage{theorem}
\usepackage[matrix,tips,frame,color,line,poly,curve]{xy}

\newtheorem{theorem}[equation]{Theorem}
\newtheorem{corollary}[equation]{Corollary}
\newtheorem{definition}[equation]{Definition}
\newtheorem{lemma}[equation]{Lemma}

\newtheorem{proposition}[equation]{Proposition}

{\theorembodyfont{\rmfamily}
\newtheorem{remarkplain}[equation]{Remark}

\newtheorem{exampleplain}[equation]{Example}

}

\renewcommand{\sec}[1]{\section{#1}
\renewcommand{\theequation}{\thesection.\arabic{equation}}
  \setcounter{equation}{0}}
\newcommand{\subsec}[1]{\subsection{#1}
\renewcommand{\theequation}{\thesubsection.\arabic{equation}}
  \setcounter{equation}{0}}

\def\danger{\begin{trivlist}\item[]\noindent%
\begingroup\hangindent=3pc\hangafter=-2
\def\par{\endgraf\endgroup}%
\hbox to0pt{\hskip-\hangindent\dbend\hfill}\ignorespaces}
\def\enddanger{\par\end{trivlist}}

\newcommand{\Gext}{\negthinspace\negthinspace\phantom{a}^\delta G}
\newcommand{\qed}{\hfill $\square$ \medskip}
\newenvironment{proof}[1][Proof]{\noindent\textbf{#1.} }{\qed}

\newcommand{\Aut}{\mathrm{Aut}}
\newcommand{\Inv}{\mathrm{Invol}}
\newcommand{\oInv}{\overline{\mathrm{Invol}}}

\newcommand{\diag}{\mathrm{diag}}

\newcommand{\Out}{\mathrm{Out}}
\newcommand{\Int}{\mathrm{Int}}
\renewcommand{\int}{\mathrm{int}}

\newcommand{\zinv}{\mathrm{inv}}
\newcommand{\SRF}{\mathrm{SRF}}
\newcommand{\Gad}{G_\mathrm{ad}}
\newcommand{\Gsc}{G_\mathrm{sc}}
\newcommand{\Zsc}{Z_\mathrm{sc}}
\newcommand{\Gbar}{\overline G}

\newcommand{\Gal}{\mathrm{Gal}}
\newcommand{\Norm}{\mathrm{Norm}}
\newcommand{\Cent}{\mathrm{Cent}}

\newcommand{\R}{\mathbb R}
\newcommand{\C}{\mathbb C}
\newcommand{\Z}{\mathbb Z}

\newcommand{\Ztwo}{\mathbb Z_2}

\newcommand{\G}{G}
\renewcommand{\H}{\mathbb H}

\newcommand{\ch}[1]{#1^\vee}

\newcommand\sigmaqc{\sigma_{\text{qc}}}
\newcommand\thetaqc{\theta_{\text{qc}}}

\newcommand{\g}{\mathfrak g}
\newcommand\inv{^{-1}}
\newcommand\wh{\widehat}

\newcommand{\Spin}{\text{Spin}}

\begin{document}
\title{Galois and $\theta$ Cohomology of Real Groups}
\author{Jeffrey Adams}
\maketitle

{\renewcommand{\thefootnote}{} 
\footnote{Supported in part by  National Science
Foundation Grant \#DMS-1317523}}

\section*{Abstract}
Suppose $G$ is a connected, reductive algebraic group.
A real form of $G$ is an antiholomorphic involution $\sigma$, so
$G(\R)=G(\C)^\sigma$ is a real Lie group.
Let$H^i_\sigma(\Gamma,G)$ be the Galois cohomology of $G$ with respect to $\sigma$.
The Cartan involution of a
real form is a holomorphic involution $\theta$ of $G$, commuting with
$\sigma$, so that $G(\R)^\theta$ is a maximal compact subgroup of $G(\R)$.
Let $H^i_\theta(\Ztwo,G)$ be the cohomology of $G$ where the action of
the nontrivial element of $\Ztwo$ is by $\theta$. The main result is
there is a canonical isomorphism $H^i_\sigma(\Gamma,G)\simeq
H^i_\theta(\Ztwo,G)$. We give several applications, including a computation of
$H^1(\Gamma,G)$ for all simple, simply connected $G$.

\sec{Introduction}

Suppose $G$ is a connected reductive algebraic group, defined over
$\R$. So, identifying $G$ with its complex points $G(\C)$, we are
given an antiholomorphic involution $\sigma$ of $G$, and
$G(\R)=G(\C)^\sigma$ is a real Lie group.  Let $\Gamma=\Gal(\C/\R)$
and write $H^i(\Gamma,G)$ for the Galois cohomology of $G$.  If we
want to specify how the nontrivial element of $\Gamma$ acts we will
write $H^i_\sigma(\Gamma,G)$.  The real forms of $G$, which are inner
to $\sigma$ (see Section \ref{s:galois}), are parametrized by
$H^1(\Gamma,\Gad)$ where $\Gad$ is the adjoint group.

On the other hand Cartan classified the real forms of $G$ in terms of
holomorphic involutions as follows.  Associated to $\sigma$ is a
Cartan involution $\theta$ of $G$. This is a holomorphic involution,
commuting with $\sigma$, such that $K(\R)=G(\R)^\theta$ is a maximal
compact subgroup of $G(\R)$. Then $K(\R)$ is the real points of the
complex, reductive, (possibly disconnected) group $K=G^\theta$.
Conversely $\sigma$ is determined by $\theta$, and 
the real forms inner to $\sigma$  can also be parametrized by  
involutions inner to $\theta$. 
 See Example \ref{ex:realforms} for
a precise statement.

Let $H^i(\Ztwo,G)$, or $H^i_\theta(\Ztwo,G)$ be the group cohomology of $G$
where the nontrivial element of $\Ztwo=\Z/2\Z$ acts by $\theta$.  Then
conjugacy classes of involutions which are inner to $\theta$ are
parametrized by $H^1_\theta(\Ztwo,\Gad)$.  Thus the equivalence
of the two classifications of real forms amounts to an isomorphism
$H^1_\sigma(\Gamma,\Gad)\simeq H^1_\theta(\Ztwo,\Gad)$.

It is natural to ask if the same isomorphism holds in general.

\begin{theorem}
\label{t:main}
Suppose $G$ is defined over $\R$, given by an antiholomorphic involution $\sigma$. 
Let $\theta$ be a 
corresponding Cartan involution. 
Then there is a canonical isomorphism $H^1_\sigma(\Gamma,G)\simeq H^1_\theta(\Ztwo,G)$. 
\end{theorem}

The interplay between the $\sigma$ and $\theta$ pictures plays a
fundamental role in the structure and representation theory of real
groups, going back at least to Harish Chandra's formulation of the
representation theory of $G(\R)$ in terms of $(\g,K)$-modules.  The
theorem is an aspect of this, and we give several applications.

First we use this to give a simple proof of a well known result about
two versions of the rational Weyl group (Proposition \ref{p:W}).
Similarly in Section \ref{s:cartans} we give a simple proof of a
result of Matsuki: there is a bijection between $G(\R)$-conjugacy
classes of Cartan subgroups of $G(\R)$ and $K$-conjugacy classes of
$\theta$-stable Cartan subgroups of $G$ \cite{matsuki}.

The  fact that $H^1(\Gamma,\Gad)$ parametrizes the set of  real forms of $G$ in a given inner class
reduces the computation of this
cohomology space to the classification of real forms, which  can be accomplished in  a
number of ways. We seek an analogous description
of $H^1(\Gamma,G)$.

A formula for  $H^1(\Gamma,G)$  is due to Borovoi
\cite{borovoi}.  We start with this, and modify it in two
steps. First of all we replace $H^1(\Gamma,G)$ with
$H^1(\Ztwo,G)$. This a more elementary object. For example 
$G(\R)$ is compact if and only if  $\theta=1$, in which case  $H^1(\Ztwo,G)$ is the conjugacy
classes of involutions of $G$. See Example \ref{ex:compact}.  Next, we bring
in the theory of strong real forms \cite{abv}, \cite{algorithms}. The
strong real forms of a group $G$ map surjectively to the real forms
(see Lemma \ref{l:strongreal}), and bijectively if $G$ is adjoint.

Let $Z$ be the center of $G$. Associated to 
$\sigma$ is its central invariant, denoted
$\zinv(\sigma)\in Z^\Gamma/(1+\sigma)Z$.  In addition
there is a notion of central invariant of a strong real form, which is an element of
$Z^\Gamma$. See Section \ref{s:galois} for these definitions.

\begin{theorem}[Proposition \ref{p:H1}]
\label{t:H1}
Suppose $\sigma$ is a real form of $G$.
Choose a representative $z\in
Z^\Gamma$ of   $\zinv(\sigma)\in Z^\Gamma/(1+\sigma)Z$. Then there is a bijection
$$
H^1(\Gamma,G)\overset{1-1}\longleftrightarrow\text{the set of
  strong real forms with central invariant }z
$$
\end{theorem}

This bijection is useful in both directions. On the one hand it is not
difficult to compute the right hand side, thereby computing
$H^1(\Gamma,G)$.
Over a p-adic field $H^1(\Gamma,G)=1$ if $G$ is simply connected (Kneser's theorem). 
Over $\R$ this is not the case, and we use Theorem \ref{t:main} to compute
$H^1(\Gamma,G)$ for all such groups.
See Section \ref{s:galois} and  the tables in Section
\ref{s:tables}. We used 
the Atlas of Lie Groups and
Representations software for some of these calculations.

On the other hand the notion of strong real form is
important in formulating  a precise version of the local Langlands
conjecture. In that context it would be more natural if 
strong real forms were described in terms of classical Galois cohomology.
The theorem provides such an interpretation. See Corollary
\ref{c:interpret}.

The author would like to thank Michael Rapoport for asking about the
interpretation of strong real forms in terms of Galois cohomology, and
apologizes it took so long to get back to him. 
He is also grateful to
Tasho Kaletha for several helpful discussions during the writing of this paper 
and of \cite{kaletha_rigid}, and Skip Garibaldi for a discussion of 
the Galois cohomology of the spin groups.

\sec{Preliminaries on Group Cohomology}
\label{s:prelim}

See \cite{serre_galois} for an overview of group cohomology.

For now suppose $\tau$ is an involution of an abstract group $G$.
Define $H^i(\Ztwo,G)$ to be the group cohomology space where
the nontrivial element of $\Ztwo$ acts by $\tau$.  
If $G$ is abelian these are groups and are defined for all $i\ge 0$. 
Otherwise these are pointed sets, and defined  only  for $i=0,1$.
We have
the standard identifications
$$
H^0(\Ztwo,G)=G^\tau,\,
H^1(\Ztwo,G)=G^{-\tau}/\{g\rightarrow x g\tau(x\inv)\}
$$
where $G^{-\tau}=\{g\in G\mid g\tau(g)=1\}$.  For $g\in G^{-\tau}$ let $[g]$ be 
the corresponding class in $H^1(\Ztwo,G)$.

If $G$ is abelian we also have the Tate cohomology groups $\wh H^i(\Ztwo,G)$ $(i\in\Z)$.
These satisfy
$$
\wh H^0(\Ztwo,G)=G^\tau/(1+\tau)G,\quad \wh H^1(\Ztwo,G)=H^1(\Ztwo,G),
$$
and (since $\Ztwo$ is cyclic), 
$\wh H^i(\Ztwo,G)\simeq \wh H^{i+2}(\Ztwo,G)$ for all $i$.

Now suppose $G$ is a connected reductive algebraic group. 
We say $G$ is defined over $\R$ if we are given 
a Galois action on $G$. 
This means that the nontrivial element 
of $\Gamma$ acts on $G(\C)$ by an antiholomorphic involution $\sigma$. 
Then we recover  the usual Galois cohomology 
$H^i(\Gamma,G)$.
(We identify $G$ with its complex points $G(\C)$, but sometimes 
write $G(\C)$ for emphasis.)
Then 
$
H^0(\Gamma,G)=G(\R)$, 
and if $G$ is abelian then 
$\wh H^0(\Gamma,G)=\pi_0(G(\R))$.

Let $G$ act on antiholomorphic involutions by conjugation:
$g:\sigma\rightarrow\int(g)\circ\sigma\circ\int(g\inv)$ where 
$\int(g)(x)=gxg\inv\quad (x\in G)$.

\begin{definition}
\label{d:realform}
A real form of $G$ is a $G$-conjugacy class of antiholomorphic
involutions of $G$. 
\end{definition}
See \cite{abv}.

\begin{remarkplain}
\label{r:notserre}
The standard definition of  of real forms, for example see
\cite[III.1]{serre_galois}, allows conjugation by an arbitrary
(holomorphic) automorphism of $G$, not just the inner automorphisms $\int(g)$.

For example suppose $G=SL(2,\C)\times SL(2,\C)$. This has four real
forms according to our definition: $\text{split}, \text{compact}$,
$\text{split}\times\text{compact}$ and
$\text{compact}\times\text{split}$.  If one allows conjugation by
outer automorphisms there are only three real forms, since 
the last two are the same. 

This distinction is crucial in our applications; see Example
\ref{ex:realforms} and Remark \ref{r:notserre2}.
For simple groups these two notions agree except in type $D_{2n}$. 
See \cite[Section 3]{algorithms},  \cite[Example 3.3]{snowbird} and 
Section \ref{s:adjoint}.

\end{remarkplain}

Now suppose  $\theta$ is a
holomorphic involution of $G=G(\C)$,
and let $K=K(\C)=G^\theta$.
Then 
$H^0(\Ztwo,G)=K$, and if $G$ is abelian 
$\wh H^0(\Ztwo,G)=\pi_0(K)$.

Since we will be using cohomology defined with respect to both holomorphic and anti-holomorphic 
involutions we write
$H^i(\Gamma,G)$ for Galois cohomology, and 
$H^i(\Ztwo,G)$ for  cohomology with respect to a holomorphic involution.
When there are multiple involutions being considered we will specify
these by writing $H^i_\sigma(\Gamma,G)$ or 
$H^i_\theta(\Ztwo,G)$.

Fix real form $\sigma$, and a holomorphic
involution $\theta$. 
We say $\theta$ corresponds to $\sigma$ if $\theta$ is a Cartan involution
of $G(\R)=G(\C)^\sigma$. In other words: $\sigma$ and $\theta$
commute, and
$G(\R)^\theta$ is a maximal compact subgroup of $G(\R)$. Equivalently
$\theta=\sigma\sigma_c$  where $\sigma,\sigma_c$ commute, and
$G(\C)^{\sigma_c}$ is compact.
See Example \ref{ex:realforms} for the relationship with the classification of real forms. 

\begin{exampleplain}
\label{ex:compact}
The group $G(\R)=G^\sigma$ is compact if and only if $\theta=1$. 
Then $H^1_\theta(\Ztwo,G)=\{g\in G\mid g^2=1\}/\{g\rightarrow xgx\inv\}$, i.e. 
$H^1_\theta(\Ztwo,G)$ is 
the set of conjugacy classes of involutions of $G$.
Therefore, if we  fix  a Cartan subgroup $H$,
with Weyl group $W$, then
$$
H^1_\theta(\Ztwo,G)\simeq H_2/W
$$
where $H_2=\{h\in H\mid h^2=1\}$.
See Example \ref{ex:serre1}.
\end{exampleplain}

\subsec{Twisting}
\label{s:twisting}
We make repeated use of twisting in nonabelian cohomology 
\cite[Section III.4.5]{serre_galois}. We describe how this works in our situation. 
Return to the setting of an abstract group $G$ together with an involution $\tau$, 
and consider $H^1_\tau(\Ztwo,G)$.

Suppose $g\in G$, and  $g\tau(g)$ is in the center $Z=Z(G)$ of $G$. 
Then $\tau'=\int(g)\circ\tau$ is an involution, and 
$H^i_{\tau'}(\Ztwo,G)$  is defined.
This  is not necessarily isomorphic to $H^i_\tau(\Ztwo,G)$, unless 
$g\in G^{-\tau}$, in which case the map $x\rightarrow gx$ induces an isomorphism
$H^1_\tau(\Ztwo,G)\simeq H_{\tau'}^1(\Ztwo,G)$.

\sec{Cohomology of Tori}
\label{s:torus}
Suppose $H$ is a torus,  and $\sigma$ is a   real form of $H$. 
Let $\theta$ be the Cartan involution corresponding to $\sigma$:
$\theta=\sigma\sigma_c$ where 
$\sigma_c$ is the compact real form of $H$ ($\sigma_c$ and 
$\theta$ are unique). Then $H(\R)^\theta$ is the maximal compact subgroup of $H(\R)$.

\begin{proposition}
\label{p:torus}
For all $i$ there is a  canonical isomorphism   $\wh H^i_\sigma(\Gamma,H)\simeq \wh H^i_\theta(\Ztwo,H)$.  
\end{proposition}
We only need to consider $i=0,1$.

\begin{remarkplain}
From the structure of real tori it is easy to see there are
isomorphisms as indicated, although not necessarily canonical ones. It is well known that 
$H(\R)\simeq S^{1a}\times \R^{*b}\times \C^{*c}$,
in which  case $K=K(\C)=\C^{*(a+c)}\times \Ztwo^b$. 

Then $\wh H_\sigma^0(\Gamma,H)=\pi_0(H(\R))\simeq \Ztwo^b$, 
and $\wh H^0_\theta(\Ztwo,H)=\pi_0(K)=\Ztwo^b$.

Furthermore it is easy to see  $H^1(\Gamma,H)$ and $H^1(\Ztwo,H)$ are
trivial if 
if $H(\R)=\R^*$ or $\C^*$  (for $\Gamma$ this is Hilbert's theorem
90), or $\Ztwo$ if $H(\R)=S^1$. 
Therefore $H^1_\sigma(\Gamma,H)\simeq H^1_\theta(\Ztwo,H)\simeq \Ztwo^a$. 
\end{remarkplain}

\begin{proof}
Consider the exact sequence
$$
1\rightarrow
H_2\rightarrow  
H\overset{z^2}\rightarrow  
H\rightarrow 1
$$
This gives
the long exact sequence in Tate cohomology:
$$
\wh H_\sigma^0(\Gamma,H)\overset\alpha\rightarrow\wh H_\sigma^0(\Gamma,H)\rightarrow
\wh H_\sigma^1(\Gamma,H_2)\rightarrow \wh H_\sigma^1(\Gamma,H)\overset\beta\rightarrow \wh H_\sigma^1(\Gamma,H)
$$
It is easy to see $\alpha,\beta$, which are induced by the square map, 
are trivial: $\sigma(t)=t$ implies $t^2=t\sigma(t)$, 
and $t\sigma(t)=1$ implies $t^2=(t\inv)\inv\sigma(t\inv)$.
So we have an exact sequence
$$
1\rightarrow \wh H_\sigma^0(\Gamma,H)
\rightarrow\wh H^1_\sigma(\Gamma,H_2)
\rightarrow \wh H^1_\sigma(\Gamma,H)
\rightarrow 1
$$
The analogous sequence holds for $\theta$:
$$
1\rightarrow \wh H^0_\theta(\Ztwo,H)
\rightarrow \wh H_\theta(\Ztwo,H_2)
\rightarrow \wh H_\theta(\Ztwo,H)
\rightarrow 1
$$
Let $H_c(\R)=H^{\sigma_c}$ be the  compact real form of $H$.
This is the closure of the subgroup of $H(\C)$ of all elements 
of finite order, so $H_2\subset H_c(\R)$. 
Since $\theta=\sigma\sigma_c$, $\theta$ and $\sigma$ agree on $H_2$,
so the identity map on $H_2$ induces an isomorphism $\phi:\wh H^1_\sigma(\Gamma,H_2)\simeq \wh H^1_\theta(\Ztwo,H_2)$.

It is well known, and easy to compute directly, that $\wh
H^0_\sigma(\Gamma,H)=\pi_0(H(\R))$ is  isomorphic to $\wh
H^0_\theta(\Ztwo,H)=\pi_0(H^\theta)$. An explicit calculation using 
any isomorphism
$H(\R)\simeq S^{1a}\times \R^{*b}\times\C^{*c}$ shows that the images 
of $\wh H^0_\sigma(\Gamma,H)$ in $H^1_\sigma(\Gamma,H_2)$ 
and  $\wh H^0_\theta(\Ztwo,H)$ in $H^1_\theta(\Ztwo,H_2)$ agree under the 
the isomorphism $\phi$.
In other words the left hand box of 
$$
\xymatrix{
&1\ar[r]&\wh H_\sigma^0(\Gamma,H)\ar[d]_\simeq\ar[r]&H^1_\sigma(\Gamma,
H_2)\ar[d]^{\phi}\ar[r]&H^1_\sigma(\Gamma,H)\ar[r]\ar[d]&1\\
&1\ar[r]&\wh H^0_\theta(\Ztwo,H)\ar[r]&H^1_\theta(\Ztwo,
H_2)\ar[r]&H^1_\theta(\Ztwo,H)\ar[r]&1
}
$$
commutes. It follows that there is a unique isomorphism
$H^1_\sigma(\Gamma,H)\simeq H^1_\theta(\Ztwo,H)$ making the entire diagram commute.
\end{proof}

\sec{Cohomology of reductive groups}
\label{s:G}

Now suppose $G=G(\C)$ is a connected reductive group, defined over
$\R$, with real form $\sigma$, and corresponding Cartan involution
$\theta$ (see Section \ref{s:prelim}).

If $H$ is a Cartan subgroup of $G$ let $N=\Norm_G(H)$, and let $W=N/H$.
If $H$ is $\sigma$-stable then $\sigma$ acts on $N$ and $W$.
The short exact sequence $1\rightarrow H\rightarrow N\rightarrow W\rightarrow 1$ gives rise to 
the exact cohomology sequence
\begin{equation}
\label{e:longexact}
H^0(\Gamma,H)\negthinspace\rightarrow\negthinspace
H^0(\Gamma,N)\negthinspace\rightarrow\negthinspace
H^0(\Gamma,W)\negthinspace\rightarrow\negthinspace
H^1(\Gamma,H)\negthinspace\rightarrow\negthinspace
H^1(\Gamma,N)\negthinspace\rightarrow\negthinspace
H^1(\Gamma,W)
\end{equation}
The 
third map takes $H^0(\Gamma,W)=W^\sigma$ to a subgroup of $H^1(\Gamma,H)$, and 
thereby acts by conjugation.

\begin{lemma}
\label{l:H1N}
Suppose $w\in W^\sigma$ and  $h\in H^{-\sigma}$. 
Choose $n\in N$ mapping to $w$.
Then the action of $w$ on $H^1(\Gamma,H)$ is 
$w:[h]\rightarrow [nh\sigma(n\inv)]$; this is well defined, independent of the choices involved.

The image of $H^1(\Gamma,H)$ in $H^1(\Gamma,N)$ is isomorphic to  $H^1(\Gamma,H)/W^\sigma$.
\end{lemma}

This is immediate. See \cite[I.5.5, Corollary 1]{serre_galois}.

We say a root $\alpha$ of $H$ in $G$ is imaginary, real, or complex if 
$\sigma(\alpha)=-\alpha$, $\sigma(\alpha)=\alpha$, or
$\sigma(\alpha)\ne\pm\alpha$, respectively. 
Let $W_i\subset W^\sigma$ be the Weyl group of the root system of imaginary roots.

\begin{lemma}
\label{l:Wi}
$H^1(\Gamma,H)/W^\sigma= H^1(\Gamma,H)/W_i$.
\end{lemma}
\begin{proof}
Write $W^\sigma=(W_C)^\sigma\ltimes[W_i\times W_r]$   
as in \cite[Proposition 4.16]{ic4}. Here $W_r$ is Weyl group of the real
roots, and 
$(W_C)^\sigma$ is a certain Weyl group, generated by terms of the form
$s_\alpha s_{\sigma\alpha}$ where $\alpha,\sigma\alpha$ are
orthogonal. 
It is easy to see that $W_r$ acts trivially on $H^1(\Gamma,H)$, 
and $(W_C)^\sigma$ does as well \cite[Proposition 12.16]{algorithms}.
\end{proof}

\begin{remarkplain}
\label{r:Wi}
Note that $H^1_\sigma(\Gamma,H)$ only depends on the restriction
of $\sigma$ to $H$, as does the notion of imaginary root. 
However this is not the case for  the action of $W_i$ on $H^1_{\sigma}(\Gamma,H)$
of Lemma \ref{l:H1N}, which is sensitive to the restriction of $\sigma$ to $N$.
\end{remarkplain}

We say a Cartan subgroup $H$ is  fundamental  if it is of minimal 
split rank.  Such a Cartan subgroup is the centralizer of a Cartan subgroup of the identity component of $K$.

\begin{proposition}
\label{p:sigma}
The  map
$H^1_\sigma(\Gamma,H)\rightarrow H^1_\sigma(\Gamma,G)$ factors through
the quotient by $W_i$, and  induces  an injection
\begin{equation}
\label{e:sigma}
\phi:H^1_\sigma(\Gamma,H)/W_i\hookrightarrow H^1_\sigma(\Gamma,G).
\end{equation}
If $H$ is fundamental this is an isomorphism.
\end{proposition}

Injectivity amounts to the  fact that for $\sigma$-stable
Cartan subgroups
of $G$,  conjugacy is equivalent to stable conjugacy
\cite[Corollary 2.3]{shelstad_innerforms} (see below).
This in turn 
follows from the analogous statement for parabolic subgroups (true over any
field), 
and $G(\R)$-conjugacy  of compact Cartan subgroups. 
Surjectivity for a fundamental Cartan subgroup $H$ is
in \cite[Lemma 10.2]{kottwitzStable}, For the fundamental Cartan
subgroup injectivity and surjectivity are proved in \cite{borovoi} (with $W^\sigma$ in place of $W_i$). 
We give complete proofs, for the convenience of the reader,
and because we need to repeat the arguments in the setting
of $H^1(\Ztwo,G)$.

\begin{lemma}
\label{l:reallevis}
Suppose $P,P'$ are $\sigma$-stable parabolic subgroups of $G$, 
with $\sigma$-stable Levi factors $M,M'$. 
If $M,M'$ are conjugate then they are $G(\R)$-conjugate.
\end{lemma}

\begin{proof}
We first show that if $P,P'$ are conjugate then they are
$G(\R)$-conjugate. 
The parabolic subgroups conjugate to $P$ are in bijection with $G/P$,
via the map $gP\rightarrow gPg\inv$. The $\sigma$-stable
parabolic subgroups in this set are given by $(G/P)(\R)$. 
The map $G(\R)\rightarrow (G/P)(\R)$, obtained by taking
$\sigma$-fixed points of the projection $G\rightarrow G/P$, is surjective
\cite[Theorem 4.13(a)]{borel_tits}. If $g\in G(\R)$ maps to $Q\in (G/P)(\R)$
then $Q=gPg\inv$.

Therefore, after conjugating by $G(\R)$, we may assume $P=P'$. Suppose $M,M'$ are
$\sigma$-stable Levi factors of $P$. All Levi factors of $P$ are
$U$-conjugate, where $U$ is the unipotent radical of $P$. Write
$M'=uMu\inv$ ($u\in U$). Then  $M'$ is $\sigma$-stable if and only if 
$u\inv\sigma(u)\in M$, but $U$ is $\sigma$-stable and $U\cap M=1$, so
$u=\sigma(u)$.

\end{proof}

We say two $\sigma$-stable Cartan subgroups $H,H'$ of $G$ are stably conjugate if
there exists $g\in G(\C)$ such that $gHg\inv=H'$ and $\int(g):H\rightarrow
H'$ is defined over $\R$.

\begin{lemma}[\cite{shelstad_innerforms}, Corollary 2.3]
\label{l:stablecartans}
Two $\sigma$-stable Cartan subgroups are stably conjugate if and
only if they are $G(\R)$-conjugate.
\end{lemma}

\begin{proof}
Suppose $H,H'$ are $\sigma$-stable and stably conjugate. 
Choose $g\in G$ so that $\int(g):H\rightarrow H'$ is defined over
$\R$.

Write $H=TA$ where $T$ (resp. $A$) is the maximal compact
(resp. split) subtorus of $H$. 
Let $M=\Cent_G(A)$, this is $\sigma$-stable and 
contained in a $\sigma$-stable parabolic subgroup $P$
(define the roots of $H$ in $U$ to be $\{\alpha\mid \text{Re}(\alpha(\gamma))>0\}$ for 
$\gamma$ a regular element of the Lie algebra of $A$). Define $M'\supset
H'$ similarly. 

The fact that $\int(g)|_H$ is defined over $\R$ implies
$\int(g)(M)=M'$, so by the previous Lemma choose $y\in G(\R)$ so that $\int(y)(M')=M$. 
Then $T$ and $\int(y)(T')$ are compact Cartan subgroups of the derived group
of $M$. Therefore there exists $m\in M(\R)$ such that
$\int(my)(T')=T$. Then $my\in G(\R)$ and 
$\int(my)(H')=H$.
\end{proof}

\begin{proof}[Proof of Proposition \ref{p:sigma}]
For injectivity suppose $h,h'\in H^{-\sigma}$ and $[h],[h']$ have the same image in $H^1(\Gamma,G)$.  
Thus $h=g h'\sigma(g\inv)$ for some $g\in G$. 
Let $\sigma'=\int(h)\circ\sigma$; this is a real form of $G$.
It is immediate that 
$\int(g):H\rightarrow g Hg\inv$ commutes with $\sigma'$.
By the previous lemma $gHg\inv=yHy\inv$ for some $y\in G^{\sigma'}$. 
Let $n=g\inv y\in N$. Then $n\inv h'\sigma(n)=y\inv gh'\sigma(g\inv)\sigma(y)=
y\inv h\sigma(y)=h$ (the last equality follows from the condition $y\in G^{\sigma'}$). 

It is easy to see the  image of $n$ in $W$ is contained in $W^\sigma$. By
Lemma \ref{l:Wi} we can replace $n$ with an element mapping to $W_i$.
This proves injectivity.

For surjectivity we follow \cite{borovoi}. Suppose $g\in G^{-\sigma}$. Write the Jordan decomposition of 
$g$ as $g=su$.
Then $u\inv\sigma(u\inv)=s\sigma(s)$, so
$u\sigma(u)=s\sigma(s)=1$. Then $u=v^2$ where $v\sigma(v)=1$ and $v$
commutes with $s$. It follows that $vg\sigma(v\inv)=s$, so without
loss of generality we may assume $g$ is semisimple.  
Furthermore $g\sigma(g)=1$ implies $g$ commutes with $\sigma(g)$, 
so without loss of generality $g$ is contained in a $\sigma$-stable
torus $H'$ (not necessarily the same as $H$). 

Write $H'=T'A'$ ($T'$ and $A'$ are compact and split, respectively)
and $g=ta$ accordingly.  Then $t\sigma(t)=(a\sigma(a))\inv$, so this
element is in $T'(\R)\cap A'(\R)$, which is trivial. Therefore
$a\sigma(a)=1$. Since $H^1(\Gamma,A')=1$ choose $b\in A'$ so that
$b\sigma(b\inv)=a$, and $bg\sigma(b\inv)=t$. So we  may assume
$g\in T'$. But $T'$, being a compact torus, is conjugate to a subtorus
of any fundamental torus. This proves surjectivity.
\end{proof}

The analogous result, with essentially the same proof, holds with $\theta$ in place of $\sigma$.
As in Lemma \ref{l:H1N}, $W^\theta$ acts on $H^1_\theta(\Ztwo,H)$.

\begin{proposition}
\label{p:theta}
The  map
$H^1(\Ztwo,H)\rightarrow H^1(\Ztwo,G)$ factors through
the quotient by $W_i$, and induces an injection
\begin{equation}
\label{e:theta}
\phi:H^1_\theta(\Ztwo,H)/W_i\hookrightarrow H^1(\Ztwo,G).
\end{equation}
If $H$ is fundamental this is an isomorphism.
\end{proposition}

First we first need versions of Lemmas \ref{l:reallevis} and \ref{l:stablecartans}.
Let $K=G^\theta$.

\begin{lemma}
\label{l:thetalevis}
Suppose $Q,Q'$ are $\theta$-stable parabolic subgroups of $G$, 
with $\theta$-stable Levi factors $L,L'$. 
If $L,L'$ are conjugate then they are $K$-conjugate.

Two $\theta$-stable Cartan subgroups are conjugate if and only if they
are $K$-conjugate.
\end{lemma}

\begin{proof}
The proof of the first part is  similar to that of Lemma \ref{l:reallevis}, using the
fact that $P\simeq K/K\cap P$.
The proof of Lemma
\ref{l:stablecartans} also carries over to this situation;
in this setting we use the fact that any two $\theta$-stable  split
Cartan subgroups of $L$
are $K$-conjugate.  We leave the details to the reader.  
\end{proof}

\begin{proof}[Proof of Proposition \ref{p:theta} (sketch)]
Write $H=TA$ as before, and let
  $L=\Cent_G(T)$.  This is contained in a $\theta$-stable Cartan
  subgroup $Q$. Write $Q=LV$ with $V$ unipotent. 
Injectivity follows as before
with
  $(Q,L,V,\theta)$ in place of $(P,M,U,\sigma)$.

In the proof of surjectivity write $g=ta$. Since $\theta$ acts by
inverse on $A$, and (since $A$ is connected) any element of $A$  is of the
form $b\theta(b\inv)$ for some $b\in A$,  after conjugating  by $b$ 
we may assume $g\in T$. Finally $T$ is $K$-conjugate to a subtorus 
of a fixed fundamental Cartan subgroup. 

We leave the few remaining details to the reader.
\end{proof}

\begin{lemma}
\label{l:agree}
Suppose $H$ is both $\sigma$ and $\theta$ stable. Then the actions of
$\sigma$ and $\theta$ on $W$ agree, and the isomorphism of Proposition \ref{p:torus} commutes with the
actions of $W^\sigma=W^\theta$. 

In particular $H^i_\sigma(\Gamma,W)=H^i_\theta(\Ztwo,W)$. 
\end{lemma}

\begin{proof}
Write $\theta=\sigma\sigma_c$ where $\sigma,\sigma_c$ commute and
$G_c(\R)=G^{\sigma_c}$ is compact.
Then $W$ is isomorphic to $\Norm_{G_c(\R)}(H^{\sigma_c})$, i.e. 
every Weyl group element has a
representative in $G_c(\R)$.   The result follows.
\end{proof}

\begin{proof}[Proof of Theorem \ref{t:main}]
Fix a $\sigma,\theta$-stable fundamental Cartan subgroup $H_f$ and 
consider the diagram
\begin{equation}
\label{diagram}
\xymatrix{
H^1_\sigma(\Gamma,H_f)/W_i\ar[d]_{\simeq}&&\ar[ll]_{\eqref{e:sigma}}^{\simeq}H^1_\sigma(\Gamma,G)\ar[d]^{\simeq}\\
H^1_\theta(\Ztwo,H_f)/W_i\ar[rr]_{\eqref{e:theta}}^{\simeq}&&H^1_\theta(\Ztwo,G)
}
\end{equation}

The left arrow is the one from Proposition \ref{p:torus}, together
with Lemma \ref{l:agree}.
Define the right arrow to be the composition of the other three. 
This is an isomorphism, depending only on the choice of $H_f$. 
Any two fundamental Cartan subgroups are conjugate. It is easy to see
this changes the induced map by twisted conjugation $g\rightarrow
xg\sigma(x\inv)$ and $g\rightarrow xg\theta(x\inv)$. These are
absorbed in the quotients defining the cohomology, so the resulting isomorphism
is independent of the choice of $H_f$.
\end{proof}

The isomorphism of Theorem \ref{t:main} may be described as follows.
Suppose $\gamma\in H^1(\Gamma,G)$.
Chooose $h\in H^{-\sigma}$ so that $\gamma=[h]$.
Furthermore choose, as is possible, $x\in H$ so that $h'=xh\sigma(h)\inv\in H_2$. 
Then take $\gamma$ to  $[h']\in H^1(\Ztwo,G)$.

Recall (Lemma \ref{l:H1N}) $H^1_\sigma(\Gamma,H)/W_i$ is the
image of $H^1(\Gamma,H)$ in $H^1(\Gamma,N)$. We
bring $N$ into the picture in Section \ref{s:cartans}.

\begin{remarkplain}
By Proposition \ref{p:torus}
if $G$ is  a torus, in addition to the isomorphism $H^1(\Gamma,G)\simeq H^1(\Ztwo,G)$,
we have $\wh H^0_\sigma(\Gamma,G)\simeq \wh H_\theta^0(\Ztwo,G)$, 
i.e. $\pi_0(G(\R))=\pi_0(G^\theta)$ (see Section \ref{s:torus}).
It is well known that $\pi_0(G(\R))\simeq \pi_0(G^\theta)$ for $G$ reductive.
It would be interesting if one could define ``non-commutative Tate cohomology'' 
in such a way that $\wh H_\theta^0(\Ztwo,G)=\pi_0(G^\theta)$
and $\wh H_\sigma^0(\Gamma,G)=\pi_0(G(\R))$, and periodicity holds,
so that Proposition \ref{p:torus} holds for $G$ reductive.

\end{remarkplain}

\begin{exampleplain}
\label{ex:serre1}
If $G(\R)$ is compact, $\theta$ is the identity, and $H^1(\Ztwo,G)$
is the set of conjugacy classes of involutions of $G=G(\C)$. 
This is in bijection with $H_2/W$, where $H=H(\C)$ is any Cartan subgroup
and $W$ is the Weyl group (see Example \ref{ex:compact}).

On the other hand $H^1(\Ztwo,G(\R))$ (with the trivial action)
is the set of conjugacy classes of involutions in $G(\R)$,
i.e. $H(\R)_2/W$. Since $H(\R)$ is compact this is equal to $H(\C)_2/W$. So we recover 
\cite[Theorem 6.1]{serre_galois}: $H^1(\Gamma,G)\simeq H^1(\Ztwo,G(\R))=H(\R)_2/W$.
\end{exampleplain}

\begin{exampleplain}
\label{ex:realforms}
Define
$$
\oInv(G)=\{\text{antiholomorphic involutions of }G\}
$$
and if $\sigma\in \oInv(G)$ let 
$$
\label{e:Inv}
\oInv_\sigma(G)=\{\mu\in \oInv(G)\mid \mu\inv\circ\sigma\text{ is an inner automorphism of }G\}.
$$
We view $\oInv_\sigma(G)$ as a pointed set with distinguished element $\sigma$.

By definition the set of  real forms of $G$
 is  $\oInv(G)/G$ (with $G$
acting by conjugation by inner automorphisms), and $\oInv_\sigma(G)/G$
is the set of real forms  inner to $\sigma$ (see Section
\ref{s:galois}).

There is a canonical isomorphism (of pointed sets)
$$
H^1_\sigma(\Gamma,\Gad)\overset{1-1}\longleftrightarrow \oInv_\sigma(G)/G.
$$
The map takes $[g]$ to 
the conjugacy class of $\int(g)\circ\sigma$.
Thus we have 
\begin{equation}
\label{e:H1realforms}
H^1_\sigma(\Gamma,\Gad)\overset{1-1}\longleftrightarrow
\{\text{equivalence classes of real forms inner to } \sigma\}
\end{equation}

\begin{remarkplain}
\label{r:notserre2}
This statement is true for our Definition of real forms. See Definition \ref{d:realform} and 
Remark \ref{r:notserre}. If we were to use  the usual definition, for example \cite{serre_galois}, we 
would replace
left hand side with the image of the map to $H^1_\sigma(\Gamma,\Aut(G))$. 
\end{remarkplain}

Now
let $\Inv(G)$ be the holomorphic involutions of $G$, and let  $\Inv_\theta(G)$  be those which are inner to $\theta$. 
Then
there is a canonical isomorphism 
\begin{equation}
\label{e:H1algad}
H^1_\theta(\Ztwo,\Gad)\overset{1-1}\longleftrightarrow \Inv_\theta(G)/G,
\end{equation}
taking $[g]$ to $\int(g)\circ\theta$. 
We have a commutative diagram:
\begin{equation}
\label{e:cartansquare}
\xymatrix{
H^1_\sigma(\Gamma,\Gad)\ar@{<->}[d]_{\text{Theorem }
  \ref{t:main}}\ar@{<->}[r]&\oInv_\sigma(G)/G\ar@{=}[r]\ar@{<->}[d]&\text{\{real
  forms of }G\text{ inner to }\sigma\}\\
H^1_\theta(\Ztwo,\Gad)\ar@{<->}[r]&\Inv_\theta(G)/G
}
\end{equation}
All the arrows are isomorphisms (of pointed sets), and the equality is
a definition. The right hand vertical arrow is Cartan's description of
real forms in terms of holomorphic involutions. 
\end{exampleplain}

\begin{exampleplain}
\label{ex:sl2}
Suppose $G=PSL(2,\C)$. This has two real forms, $PGL(2,\R)\simeq SO(2,1)$ and $SO(3)$.
Since $G$ is adjoint $|H^1(\Gamma,G)|=2$ for either real form.

Now let $G=SL(2,\C)$. From  Example \ref{ex:serre1} if $G(\R)=SU(2)$ then $|H^1(\Gamma,G)|=2$. 
On the other hand if $G(\R)=SL(2,\R)$ then 
it is well known that $H^1(\Gamma,G)=1$.
To see this using Theorem \ref{t:main},
take 
$H$ to be the diagonal Cartan subgroup, and $\theta_c=1$, $\theta_s=\int(\diag(i,-i))$
(the Cartan involutions for $SU(2)$ and $SL(2,\R)$, respectively). 
In both cases $H_2=\pm I$. 
What is different  is the twisted action of $W_i$, which is trivial if $\theta=\theta_c$, 
whereas 
if $g$ represents the nontrivial element of the Weyl group
then  $gI\theta_s(g\inv)=-I$. 

Note that, in contrast to the adjoint case, although $SL(2,\R)$ and $SU(2)$ are inner forms of each other, 
their cohomology is different. See Lemma \ref{c:clarify}.
\end{exampleplain}

\begin{corollary}
\label{c:Mf}
Suppose $H_f$ is a fundamental Cartan subgroup.
Let $A_f$ be the 
the identity component of the (complex)
maximal split subtorus and let $M_f=\Cent_G(A_f)$. 
Then
$$
H^1_\sigma(\Gamma,G)\simeq H^1_\sigma(\Gamma,M_f)\simeq
H^1_\theta(\Ztwo,M_f)\simeq
H^1_\theta(\Ztwo,G).
$$
\end{corollary}
Note that $A_f\subset Z\Leftrightarrow M_f=G\Leftrightarrow\text{ the derived group of }G$ is of equal
rank. 

\sec{Weyl groups and conjugacy of Cartan subgroups}
\label{s:cartans}

We continue in the setting of the previous section, with a Galois action $\sigma$ and a
corresponding Cartan involution $\theta$.

We first give short proof of a well known fact about Weyl groups 
\cite[Proposition 1.4.2.1]{warner_I}, \cite[Definition 0.2.6]{vogan_green}.

\begin{proposition}
\label{p:W}
Let $N$ be the normalizer of a $\theta,\sigma$-stable Cartan subgroup $H$.
Then there is a canonical isomorphism
\begin{equation}
\label{e:norm}
\Norm_{G(\R)}(H(\R))/H(\R)\simeq \Norm_{K}(H)/(H\cap K)
\end{equation}
\end{proposition}

\begin{proof}
It is easy to see $\Norm_{G(\R)}(H(\R))=N^\sigma$, so the left hand side of \eqref{e:norm} is $N^\sigma/H^\sigma$. 
The exact sequence $1\rightarrow H\rightarrow N\rightarrow W\rightarrow 1$ gives a short exact sequence
$1\rightarrow N^\sigma/H^\sigma\rightarrow W^\sigma\rightarrow H^1_\sigma(\Gamma,H)$.
Similarly the right hand side of \eqref{e:norm} is $N^\theta/H^\theta$, and we have
$1\rightarrow N^\theta/H^\theta\rightarrow W^\theta\rightarrow H^1_\theta(\Ztwo,H)$.
Together we have the commutative diagram

\begin{equation}
\xymatrix{
1\ar[r]&N^\sigma/H^\sigma\ar[d]\ar[r]&W^\sigma\ar[r]\ar[d]^=&H^1_\sigma(\Gamma,H)\ar[d]^{\simeq}\\
1\ar[r]&N^\theta/H^\theta\ar[r]&W^\theta\ar[r]&H^1_\theta(\Ztwo,H)
}
\end{equation}
which gives the isomorphism $N^\sigma/H^\sigma\simeq N^\theta/H^\theta$ as desired.
\end{proof}

We turn now to Matsuki's result on conjugacy of Cartan subgroups.

\begin{lemma}
\label{l:H1N}
Let $N$ be the normalizer of a $\theta,\sigma$-stable Cartan subgroup $H$.
Let $\Phi$ be the isomorphism of Theorem \ref{t:main}. 
There  is a canonical isomorphism 
$\Psi$ making the following diagram commute:

\begin{equation}
\label{e:H1N}
\xymatrix{
H^1_\sigma(\Gamma,N)\ar[d]_{\Psi}\ar[r]&H^1_\sigma(\Gamma,G)\ar[d]^\Phi\\
H^1_\theta(\Ztwo,N)\ar[r]&H^1_\theta(\Ztwo,G)
}
\end{equation}
\end{lemma}
\medskip

\begin{proof}
The map $N\rightarrow W$ induces a map $H^1_\sigma(\Gamma,N)\rightarrow H^1_\sigma(\Gamma,W)$. 
For $\xi\in H^1_\sigma(\Gamma,W)$ let $F_\sigma(\xi,N)$ be the fiber over $\xi$,
so $H^1_\sigma(\Gamma,N)$ is the disjoint union of the $F_\sigma(\xi,N)$. 
For $\xi\in H^1_\theta(\Ztwo,W)$ define $ F_\theta(\xi,N)\subset H^1_\theta(\Ztwo,N)$ similarly.

We proceed one fiber at a time. We identify $H^1_\sigma(\Gamma,W)$ and $H^1_\theta(\Ztwo,W)$ (cf. Lemma \ref{l:agree}).
It is enough to show that for all $\xi\in H^1_\sigma(\Gamma,W)$ 
such that $F_\sigma(\xi,N)$ is nonempty, there is an isomorphism $\Psi_\xi$ making this diagram commute:
\begin{equation}
\xymatrix{
F_\sigma(\xi,N)\ar[d]_{\Psi_\xi}\ar[r]^{\iota_\sigma}&H^1_\sigma(\Gamma,G)\ar[d]^{\Phi}\\
 F_\theta(\xi,N)\ar[r]^{\iota_\theta}&H^1_\theta(\Ztwo,G)
}
\end{equation}
It is enough to show the two fibers on the left are isomorphic, and
the images of $\Phi\circ\iota_\sigma$ and $\iota_\theta$ agree.

First take $\xi=1$. 
By Proposition \ref{p:sigma}, and the discussion preceding it, 
$$
F_\sigma(1,N)\simeq H^1_\sigma(\Gamma,H)/W_i,\,
F_\theta(1,N)\simeq H^1_\theta(\Ztwo,H)/W_i.
$$
These are isomorphic by Proposition \ref{p:torus}
and Lemma \ref{l:agree}.  The
commutativity of the diagram is clear since the horizontal maps are induced by inclusion.

We treat the general fiber by twisting.
Suppose $w\in W^{-\sigma}$ and assume $F_\sigma([w],N)$ is nonempty.
Therefore there exists 
$y\in N^{-\sigma}$ mapping to $w$. 
Let $\sigma'=\int(y)\circ\sigma$.
Twisting by $y$ (see Section \ref{s:twisting}) defines
an isomorphism $H^1_{\sigma'}(\Gamma,N)\simeq H^1_\sigma(\Gamma,N)$, taking 
$F_{\sigma'}(1,N)$ to $F_\sigma(w,N)$.
It also gives an isomorphism $H^1_{\sigma'}(\Gamma,G)\simeq
H^1_\sigma(\Gamma,G)$.

Similar comments apply in  the $\theta$ setting. Putting these
together we have the following commutative diagram, where the central
square comes from the previous discussion with $\sigma',\theta'$ in
place of $\sigma,\theta$, and $\Psi_\xi$ is defined to make the diagram commute.

 \begin{equation}
 \xymatrix{
 F_\sigma([w],N)\ar[d]_{\Psi_\xi}\ar[r]^{\simeq}&F_{\sigma'}(1,N)\ar[d]\ar[r]^{\iota_\sigma}&H^1_{\sigma'}(\Gamma,G)\ar[d]
 \ar[r]^{\simeq}&H^1_\sigma(\Gamma,G)\ar[d]^{\Phi}\\
  F_\theta([w],N)\ar[r]^{\simeq}&F_{\theta'}(1,N)\ar[r]^{\iota_\theta}&H^1_{\theta'}(\Ztwo,G)
 \ar[r]^{\simeq}&H^1_\theta(\Ztwo,G)\\
 }
 \end{equation}
\end{proof}

Let $H^1_\sigma(\Gamma,N)_0$ be the kernel of the map 
$H^1_\sigma(\Gamma,N)\rightarrow H^1_\sigma(\Gamma,G)$. 
It is well known that  that the set of $G(\R)$-conjugacy classes of 
Cartan subgroups defined over $\R$ is parametrized by
 $H^1_\sigma(\Gamma,N)_0$ as follows.
Suppose  $n\in N^{-\sigma}$, and 
$[n]$ is trivial in $H^1_\sigma(\Gamma,G)$.
Write $n=g\sigma(g\inv)$ for some $g\in G$; the map takes $[n]$ to
$gHg\inv$. 
It is straightforward to see this gives the stated bijection.

Analogous statements hold for $\theta$: the kernel 
$H^1_\theta(\Ztwo,H)_0$ of 
$H^1_\theta(\Ztwo,H)\rightarrow H^1_\theta(\Ztwo,G)$
parametrizes the $\theta$-stable Cartan subgroups.

Matsuki's result on Cartan subgroups (\cite{matsuki},\cite[Proposition
6.18]{av1}) is now immediate.

\begin{proposition}
\label{p:matsuki}
There is a bijection between $G(\R)$-conjugacy classes of $\sigma$-stable 
Cartan subgroups, and $K$-conjugacy classes of $\theta$-stable 
Cartan subgroups.
\end{proposition}

\begin{proof}
The isomorphism $\Psi$ of Proposition \ref{l:H1N} restricts to 
an isomorphism $H^1_\sigma(\Gamma,N)_0\simeq  H^1_{\theta}(\Ztwo,N)_0$.
\end{proof}

As a corollary of the proof of Lemma  \ref{l:H1N} we obtain a
description of $H^1(\Gamma,N)$. 

\begin{proposition}
Suppose $H$ is a $\sigma$-stable Cartan subgroup. Let
$S$ be a set of representatives 
of the $G(\R)$-conjugacy classes of  $\sigma$-stable Cartan
subgroups. For $H'\in S$ let  $W(H')_i$ be the imaginary Weyl group of $H'$. 
$$
H^1_{\sigma}(\Gamma,N)\simeq \bigcup_{H'\in S} H^1_{\sigma}(\Gamma,H')/W(H')_i
$$
If $H$ is $\theta$-stable the analogous result holds:
$$
H^1_\theta(\Ztwo,N)\simeq \bigcup_{H'\in S} H^1_\theta(\Ztwo,H')/W(H')_i
$$
Using the bijection of Proposition \ref{p:matsuki} these two sets are termwise isomorphic.
\end{proposition}

\sec{Strong real forms and   $H^1(\Gamma,G)$}
\label{s:galois}

Suppose $G$ is defined over $\R$. Part of the long exact cohomology sequence, associated to the exact sequence $1\rightarrow
Z\rightarrow G\rightarrow\Gad\rightarrow 1$, is
\begin{equation}
\label{e:exact1}
H^1(\Gamma,G)\rightarrow H^1(\Gamma,\Gad)\rightarrow H^2(\Gamma,Z).
\end{equation}
Recall (Example \ref{ex:realforms})
$$
H^1(\Gamma,\Gad)=\oInv_\sigma(G)/G=\{\text{real forms of $G$ inner to $\sigma$}\}.
$$
We seek a similar description of 
$H^1(\Gamma,G)$. This is straightforward if the 
the map to $H^1(\Gamma,\Gad)$ in
\eqref{e:exact1}
is surjective. In general we need to to replace $H^1(\Gamma,G)$ with a
bigger space which maps surjectively to $H^1(\Gamma,\Gad)$. This is
provided by the theory of strong real forms \cite{abv}. We 
follow the equivalent version described in \cite{algorithms}
(see Remark \ref{rem:equivalent}).
We work in the context of $H^1(\Ztwo,G)$, and then use Theorem
\ref{t:main} to state the results in terms of $H^1(\Gamma,G)$. 

For now assume we are given  only a complex reductive group $G$.
We make use of the exact sequence
\begin{equation}
\label{e:exact}
1\rightarrow\Int(G)\rightarrow\Aut(G)\rightarrow\Out(G)\rightarrow 1
\end{equation}
where $\Aut(G)$ is the (holomorphic) automorphisms of $G$, $\Int(G)$ are the inner ones,
and $\Out(G)=\Aut(G)/\Int(G)$.
We say two automorphisms are inner to
each other, or in the same inner class, if 
they have the same image in $\Out(G)$.
Thus an inner class is determined by an involution $\tau\in\Out(G)$,
and we refer to this as the inner class of $\tau$.
Note that the action of $\Aut(G)$ restricted to  $Z$ factors to $\Out(G)$.

We say two real forms $\sigma,\sigma'$ are in the same inner class if
$\sigma'\circ\sigma\inv\in \Int(G)$. Equivalently, if $\theta,\theta'$ are
corresponding Cartan involutions, then $\theta$ and $\theta'$ are in
the same inner class. See example \ref{ex:realforms}.

We take as our starting point an involution $\tau\in\Out(G)$, which determines 
an inner class of real forms. Let 
$\Inv_\tau(G)\subset \Inv(G)$ be the involutions in the
inner class of $\tau$. 
There are two natural choices for a basepoint making $\Inv_\tau(G)/G$ 
a pointed set.
 One is the quasisplit (most split) real form in the
inner class. Because of our focus on $\theta$, rather than $\sigma$,
we prefer to choose the quasicompact (most compact) form:
a real form is said to be quasicompact if its Cartan involution 
fixes a pinning datum $(H,B,\{X_\alpha\})$ 
\cite{springer_book}, \cite[Section 2.1]{algorithms}. 
Each inner class contains a unique quasicompact real form, whose Cartan involution is denoted $\thetaqc$.
The inner class is said to be of equal rank, or compact, if any
of the following equivalent conditions hold: $\tau=1$; $\thetaqc=1$; the inner
class contains the compact form of $G$;
$\text{rank}(G)=\text{rank}(G^\theta)$ for all $\theta\in\Inv_\tau(G)$.

We fix once and for all a pinning datum
$(H,B,\{X_\alpha\})$ 
This defines a
splitting of \eqref{e:exact}, taking $\Out(G)$ to 
the elements of $\Aut(G)$ which preserve the pinning. 
Then $\thetaqc$ is the image of $\tau$ under the splitting. 
The associated real form is quasicompact.

Let 
$$
\Gext=G\rtimes\Ztwo=\langle G,\delta\rangle\quad (\delta^2=1,\, \delta g\delta\inv=\thetaqc(g)).
$$

\begin{definition}
\label{d:strongreal}
A strong involution is an element $x\in \G\delta$ satisfying $x^2\in
Z$. A strong real form is a $G$-conjugacy class of strong involutions.  
Let $\SRF_\tau(G)$ be the set of strong real forms (in the inner class
of $\tau$):
\begin{equation}
\label{e:srf}
\SRF_\tau(G)=\{x\in G\delta\mid x^2\in Z(G)\}/G.
\end{equation}

Set $\zinv(x)=x^2\in
Z^\tau$. This is invariant under conjugation and so defines 
a map $\zinv:\SRF_\tau(G)\rightarrow Z^\tau$. We refer to $\zinv$ as the central invariant of a strong real form.
\end{definition}

If $x$ is a strong involution define $\theta_x\in\Inv(G)$ by 
$\theta_x(g)=xgx\inv$ ($g\in G$).
The map  $x\rightarrow\theta_x$ factors to a surjection
$$
\begin{aligned}
\SRF_\tau(G)\twoheadrightarrow \Inv_\tau(G)/G=\{\text{real forms in
  the inner class of }\tau\}.
\end{aligned}
$$
The surjectivity statement of Proposition \ref{p:theta} amounts to the same being true when
restricted to strong involutions in $H\delta$.

\begin{lemma}
\label{l:strongreal}
If $\theta\in\Inv_\tau(G)$ then $\theta=\int(x)$ for
some strong involution $x$. A conjugate of $\theta$ is equal to $\theta_x$ for some $x\in H\delta$. 
\end{lemma}

Therefore is enough to compute 
$H^1_{\theta_x}(\Ztwo,G)$ for all strong involutions $x\in H\delta$.
By Proposition \ref{p:theta} (using the fact that $H$ is a fundamental Cartan subgroup with respect to $\theta_x$)
\begin{subequations}
\renewcommand{\theequation}{\theparentequation)(\alph{equation}}  
\begin{equation}
\begin{aligned}
H^1_{\theta_x}(\Ztwo,G)&\simeq
H^1_{\theta_x}(\Ztwo,H)/W_i\\
&=[H^{-\theta_x}/(1+\theta_x)H]/W_i
\end{aligned}
\end{equation}
Note  that, since $\theta_x|_H=\thetaqc$, the numerator is
is independent of $x$,
although the action of $W_i$ depends on $x$ (see Remark \ref{r:Wi}).

We rewrite this expression by twisting by $x$.
This is analogous to twisting of cohomology as in
Section \ref{s:twisting} 
(that this is more than an analogy is explained in Section 
\ref{s:rigid}). 
Let $z=\zinv(x)\in Z^\tau$, and consider the map  $h\rightarrow hx\in H\delta$. 
It is easy to see this gives a bijection between (a) and
\begin{equation}
[\{y\in H\delta\mid y^2=z\}/\negthinspace\sim\negthinspace H]/W_i
\end{equation}
The main point is that now $H$, and especially $W_i$, 
are  acting on $H\delta$ by ordinary conjugation.
We make the  action of $W_i$ explicit. 
If $w\in W_i$, choose $n\in
N$ mapping to $w$, and take $y\in H\delta$ to $nyn\inv$. 
Writing $y=h\delta$, $nyn\inv=(nhn\inv)n\delta(n\inv)\delta$, which is
in $H\delta$ since $w\in W_i$. Furthermore this is easily seen to be
independent of the choice of $n$.
\end{subequations}
This proves the following Proposition.

\begin{lemma}
\label{l:H1Z2}
Suppose $x\in H\delta$ is a strong involution. Let $z=\zinv(x)\in Z^\tau$. Then there
is a bijection
$$
H^1_{\theta_x}(\Ztwo,G)\leftrightarrow[\{y\in H\delta\mid
y^2=z\}/\negthinspace\sim\negthinspace H]/W_i
$$
The right hand side is precisely the strong real forms with central invariant $z$.
This only depends on $z$, 
and its order only depends only on the image of $z$ in 
$Z^\tau/(1+\tau)Z$.
\end{lemma}
Tracing through the construction we see the map from right to left takes 
$y\in H\delta$ to  $[yx\inv]$. 

This result is not optimal because, given a Cartan involution $\theta$,
to compute $H^1_\theta(\Ztwo,G)$ we need to choose a strong
involution $x$
so that $\theta$ is conjugate to $\theta_x$.
On the other hand the right hand side of the Lemma only depends on
$z\in Z^\tau$. What is missing is a  definition of invariant of a real
form (compare Definition \ref{d:strongreal}).

\begin{definition}
\label{d:invariant}
Define $\zinv:\Inv_\tau(G)/G\rightarrow Z^\tau/(1+\tau)Z$ by 
the composition of maps:
\begin{equation}
\label{e:invariant}
\Inv_\tau(G)/G\leftrightarrow H^1_{\thetaqc}(\Ztwo,\Gad)\rightarrow
H^2_{\thetaqc}(\Ztwo,Z)\simeq Z^\tau/(1+\tau)Z.
\end{equation}
We refer to $\zinv$ as the central invariant of a real form.
\end{definition}

Note that the central invariant of the quasicompact real form is the identity.

For the first map, it is straightforward to see that the map $g\rightarrow \int(g)\circ\thetaqc$ induces a bijection
$
H^1_{\thetaqc}(\Ztwo,\Gad)
\longleftrightarrow
\Inv_\tau(G)/G
$.
The second  map  is from the long exact cohomology sequence
associated to the exact sequence $1\rightarrow Z\rightarrow
G\rightarrow G_{ad}\rightarrow 1$.
Since $\Ztwo$ is cyclic
there is an isomorphism in Tate cohomology
$H^2_\theta(\Ztwo,Z)=\wh
H^2_\theta(\Ztwo,Z)\simeq \wh H^0_\theta(\Ztwo,Z)$,  and the final isomorphism
is the standard description of this cohomology.

Note that the invariant of a real form is an element of
$Z^\tau/(1+\tau)Z$, whereas the invariant of a strong real form (Definition \ref{d:strongreal}) is an
element of $Z^\tau$. Now we can restate Lemma \ref{l:H1Z2}.

\begin{lemma}
Suppose $\theta\in\Inv(G)$.
Choose a representative $z\in Z^\tau$  of $\zinv(\theta)\in Z^\tau/(1+\tau)Z$. 
Then  there is a bijection
$$
H^1_\theta(\Ztwo,G)\longleftrightarrow \text{the strong real forms with central invariant }z
$$
\end{lemma}

Finally we pass back to the antiholomorphic picture to get Theorem \ref{t:H1}.
We first need a version of invariant in this setting.
Let $\sigmaqc$ be an anitholomorphic involution corresponding to 
$\thetaqc$, so $G(\R)=G^{\sigmaqc}$ is quasicompact.
Every real form in this inner class has the same restriction to $Z$,
which we denote $\sigma$. 
As in \eqref{e:invariant}
define a map  $\zinv:\oInv_\tau(G)/G\rightarrow
Z^\Gamma/(1+\sigma)Z$
by:
\begin{equation}
\label{e:sigmainvariant}
\oInv_\tau(G)/G\leftrightarrow
H^1_{\sigmaqc}(\Gamma,\Gad)\rightarrow
H^2_{\sigmaqc}(\Gamma,Z)\simeq Z^\Gamma/(1+\sigma)Z.
\end{equation}
This takes the quasicompact form to the identity.

\begin{lemma}
There is a canonical bijection between  $Z^\tau/(1+\tau)Z$ and
$Z^\Gamma/(1+\sigma)Z$.
This makes the following diagram commute:
$$
\xymatrix{
H^1_{\thetaqc}(\Ztwo,\Gad)\ar[r]^{\eqref{e:invariant}}\ar[d]_{\text{Theorem } \ref{t:main}}& Z^\tau/(1+\tau)Z\ar[d]\\
H^1_{\sigmaqc}(\Gamma,\Gad)\ar[r]_{\eqref{e:sigmainvariant}}& Z^\Gamma/(1+\sigma)Z
}.
$$
\end{lemma}

This comes down to the fact that we can choose representatives
of $Z^\tau/(1+\tau)Z$ and $Z^\Gamma/(1+\sigma)Z$ of finite order, and
$\theta$ and $\sigma=\theta\sigma_c$ agree on these.

\begin{proposition}
\label{p:H1}
Suppose $\sigma$ is a real form of $G$, and choose a representative 
$z\in Z^\Gamma$ of $\zinv(\sigma)\in Z^\Gamma/(1+\sigma)Z$. 
Then there is a bijection
$$
H^1(\Gamma,G)\longleftrightarrow\text{the set of strong real forms of central invariant }z.
$$
\end{proposition}

\begin{corollary}
\label{c:equalrank}
Suppose $G(\R)$ is an equal rank real form.
Choose $x\in G$ so that $\Cent_G(x)$ is a complexified maximal compact
subgroup of $G(\R)$, and let $z=x^2\in Z(G)$. Then
$$
H^1(\Gamma,G)\longleftrightarrow \text{the set of conjugacy classes
  of $G$ with square equal to $z$}
$$
If $H$ is any Cartan subgroup, with Weyl group $W$, then this is equal to
\begin{equation}
\label{e:h2z}
\{h\in H\mid h^2=z\}/W
\end{equation}
\end{corollary}

\begin{exampleplain}
\label{ex:serre2}
Taking $x=z=I$ gives $G(\R)$ compact and recovers \cite[Theorem
6]{serre_galois}: $H^1(\Gamma,G)$ is the set of  conjugacy classes of
$G$ of involutions of $G$. See Example \ref{ex:serre1}.
\end{exampleplain}

\begin{exampleplain}
Let $G(\R)=Sp(2n,\R)$. We can take $x=\diag(iI_n,-iI_n)$, $z=-I$. 
It is easy to see that every element of $G$ whose square is $-I$ is conjugate to $x$.
This
gives the classical result $H^1(\Gamma,G)=1$, which 
is equivalent to the classification of nondegenerate symplectic forms \cite[Chapter 2]{platonov_rapinchuk}.
\end{exampleplain}

\begin{exampleplain}
Suppose $G(\R)=SO(Q)$, the isometry group of a nondegenerate real quadratic
form.
Suppose $Q$ has signature $(p,q)$.
If $pq$ is even we can take $z=I$, Corollary \ref{c:equalrank} applies, 
and the set \eqref{e:h2z} is equal to $\{\diag(I_r,-I_s)\mid r+s=p+q;\,s\text{ even}\}$.

Suppose $p$ and $q$ are odd. Apply Corollary \ref{c:Mf} with  $M_f(\R)=SO(p-1,q-1)\times GL(1,\R)$.
By the previous case we conclude $H^1(\Gamma,G)$ is parametrized by $\{\diag(I_r,I_s)\mid r+s=p+q-2;\, r,s\text{ even}\}$. 
Adding $(1,1)$ this is the same as $\{\diag(I_r,-I_s)\mid r+s=p+q;\, s\text{ odd}\}$. 

In all cases we recover the classical fact that $H^1(\Gamma,G)$ parametrizes the set of equivalence classes of 
quadratic forms of the same dimension and discriminant as $Q$ \cite[Chapter 2]{platonov_rapinchuk},
\cite[III.3.2]{serre_galois}.
\end{exampleplain}

\begin{exampleplain}
\label{ex:spin}
Now suppose $G(\R)=\Spin(p,q)$, which is  a two-fold cover of
the identity component of 
$SO(p,q)$. 
A  calculation similar to that in the previous example 
shows that 
$|H^1(\Gamma,\Spin(p,q))|=\lfloor\frac{p+q}4\rfloor+\delta(p,q)$
where $0\le \delta(p,q)\le 3$ depends on $p,q\pmod4$.
See Section \ref{s:simply}.

Skip Garibaldi pointed out this result can also be derived from 
the long exactly cohomology sequence associated to 
the exact sequence $1\rightarrow\Ztwo\rightarrow \Spin(n,\C)\rightarrow SO(n,\C)\rightarrow 1$;
the preceding result; the fact that $SO(p,q)$ 
is connected if $pq=0$ and otherwise has two connected components; 
and a calculation of the image of the map from
$H^1(\Gamma,Spin(n,\C))\rightarrow H^1(\Gamma,SO(n,\C))$.
See \cite[after (31.41)]{bookofinvolutions}, \cite[III3.2]{serre_galois} and also section \ref{s:fibers}.
The result is:

\medskip

\noindent $|H^1(\Gamma,\Spin(Q))|$ equals the number of  quadratic forms having the
same dimension, discriminant, and Hasse invariant as $Q$ with each (positive
or negative) definite form counted twice.
\end{exampleplain}

\begin{remarkplain}
\label{rem:equivalent}
In \cite{abv} and 
\cite{vogan_local_langlands} strong real forms are defined 
in terms of the Galois action, as opposed to the Cartan involution 
as  in \cite{algorithms}  (and elsewhere, including \cite{bowdoin}).
The preceding discussion shows that these two theories are indeed equivalent. 
However the choices of basepoints in the two theories are different. 
In the Galois setting  we choose the quasisplit form, and 
in the algebraic setting  we use the quasicompact one. 

In \cite{vogan_local_langlands} the invariant of a real form
is given by \cite[(2.8)(c)]{vogan_local_langlands}.
This differs from the normalization here by multiplication 
by  $\exp(2\pi i\ch\rho)\in Z$.
Note that 
the ``pure'' rational forms, which are parametrized by $H^1(\Gamma,G)$,
include the quasisplit one 
\cite[Proposition 2.7(c)]{vogan_local_langlands}, rather than the quasicompact one.
\end{remarkplain}

\begin{remarkplain}
Kottiwtz  relates $H^1(\Gamma,G)$ to
the center of the dual group \cite[Theorem 1.2]{kottwitzStable}. This is a somewhat different type of
result. For example this result identifies the kernel of the map
from $H^1(\Gamma,G_{\text{sc}})$ to $H^1(\Gamma,G)$, but if $G$ is simply connected this gives
no information.
\end{remarkplain}

Contrary to the adjoint case, inner forms do not necessarily have isomorphic cohomology, as was illustrated in Example 
\ref{ex:sl2}. Proposition \ref{p:H1} clarifies the situation. 
See \cite[Section I.5.7, Remark 1]{serre_galois}.

\begin{corollary}
\label{c:clarify}
Suppose $\sigma,\sigma'$ are inner forms of $G$. If $\zinv(\sigma)=\zinv(\sigma')$ then 
$H^1_\sigma(\Gamma,G)\simeq H^1_{\sigma'}(\Gamma,G)$. 
\end{corollary}

\sec{Relation with Rigid Rational forms}
\label{s:rigid}

The space of strong rational forms can naturally be thought of 
as a cohomological object. From this point of view the proof of Lemma 
\ref{l:H1Z2} amounts to the standard twisting argument.

Vogan \cite[Problem 9.3]{vogan_local_langlands} has conjectured that
there should be a notion of strong rational form in the p-adic case,
generalizing the real case, and gave a number of
properties this definition should satisfy.  Kaletha has found a
definition which satisfies these conditions \cite{kaletha_rigid}, and
the relationship between Galois cohomology and rigid rational forms carries over to that setting. 
We confine ourselves to the real case, and refer the reader to
\cite{kaletha_rigid} for more details, and the p-adic case.

Recall $\SRF_\tau(G)=\{x\in G\delta\mid x^2\in Z\}/G$. 
It is easy to see there is an exact sequence of pointed sets

\begin{equation}
\label{e:exactsrf}
1\rightarrow H^1_{\thetaqc}(\Ztwo,G)\rightarrow \SRF_\tau(G)\overset{\zinv}\longrightarrow Z^\tau
\end{equation}
The first map takes $[g]$ $(g\in G^{-\thetaqc})$ 
to $g\delta$, and $ \zinv(x)=x^2$ 
(see the discussion following
Definition \ref{d:invariant}, and  Definition \ref{d:strongreal}).

By definition the strong real forms of invariant $z\in Z^\tau$ are the fiber 
over $z$ of the invariant map, and 
the exact sequence identifies $H^1_{\thetaqc}(\Ztwo,G)$ as the strong real forms of invariant $1$. 

\begin{proposition}
Suppose the fiber over  $z\in Z^\tau$ is nonempty. 
Choose $x\in G\delta$ so that  $x^2=z$.
Then we may identify the fiber over $z$ with $ H^1_{\theta_x}(\Ztwo,G)$.
\end{proposition}
Since the fiber over $z$ is, by definition, the strong real forms of invariant $z$,
this gives Lemma \ref{l:H1Z2}.

\begin{proof}
The proof is by twisting (and is essentially equivalent to the proof of Lemma \ref{l:H1Z2}). 
For this we generalize the definition of $\SRF_\tau(G)$.
Given $\theta\in\Inv(G)$ consider the group 
$$
G\rtimes\Ztwo=\langle G,\delta_\theta\rangle\quad (\delta_\theta^2=1,\, \delta_\theta g\delta_\theta\inv=\theta(g))
$$
and
$$
\SRF(\theta,G)=\{x\in G\delta_\theta\mid x^2\in Z\}/G.
$$
We view this is a pointed set with distinguished element $\delta_\theta$.
With this notation $\SRF_\tau(G)=\SRF(\thetaqc,G)$.

Just as in \eqref{e:exactsrf} there is an
exact sequence
\begin{equation}
\label{e:exactsrf2}
1\rightarrow H^1_\theta(\Ztwo,G)\rightarrow \SRF(\theta,G)\overset{\zinv}\longrightarrow Z^\tau
\end{equation}

Suppose $x\in\G\delta_\theta$ satisfies $z=x^2\in Z$, and let $\theta'=\int(x)\circ\theta$. 
Unless $z=1$ twisting by $x$ is not defined at the level of $H^1(\Ztwo,G)$, so
$H^1_\theta(\Ztwo,G)$ and $H^1_{\theta'}(\Ztwo,G)$ are not necessarily isomorphic.
See Section \ref{s:twisting}. 
However twisting by $x$ makes sense within the larger set $\SRF(\theta,G)$, and 
defines an isomorphism of 
$\SRF(\theta,G)$ and $\SRF(\theta',G)$.

More precisely,   there is a natural bijection $t_x:\SRF_\tau(G)\rightarrow \SRF(\theta_x,G)$
essentially given by multiplying on the right by $x\inv$. 
We need to account for  the fact that we have two different groups 
$\langle G,{\delta_{\theta_x}}\rangle$  and $\langle G,\delta\rangle$. To be precise
the map  $G\delta\ni y\rightarrow (yx\inv)\delta_{\theta_x}\in G\delta_{\theta_x}$ induces 
a bijection $t_x$, as is easily checked. This takes $x\in\SRF_\tau(G)$ to the basepoint
$\delta_{\theta_x}\in\SRF(\delta_\theta,G)$.

We obtain the commutative diagram of (non-pointed) sets:
\begin{equation}
\xymatrix{
&&\SRF_\tau(G)\ar[r]\ar[d]_{t_x}&Z^\tau\ar[d]\\
1\ar[r]& H^1_{\theta_x}(\Ztwo,G)\ar[r]&\SRF(\theta_x,G)\ar[r]&Z^\tau
}
\end{equation}
where the final vertical arrow is multiplication by $z\inv$. 
By the standard twisting argument (or direct calculation) $t_x$ takes the fiber over $z$ to the fiber over $1$, 
which is
$ H^1_{\theta_x}(\Ztwo,G)$. 
\end{proof}

\begin{corollary}
\label{c:interpret}
Choose representatives $z_1,\dots, z_n\in Z^\tau$ for the image of $\zinv:\SRF_\tau(G)\rightarrow Z^\tau$. 
For each $z_i$ choose $x_i\in G\delta$ such that $\zinv(x_i)=z_i$. Then 
there is a  bijection
$$
\SRF_\tau(G)\longleftrightarrow\bigcup_i H^1_{\theta_{x_i}}(\Ztwo,G).
$$
\end{corollary}
(Strictly speaking the image of inv is infinite if  the center of
$G$ contains a compact torus. As in \cite{kaletha_rigid} 
or \cite[Section 13]{algorithms} the theory can be modified so this set is finite.) 

This gives an interpretation of $\SRF_\tau(G)$ in classical cohomological terms.
A similar statement holds in the p-adic case \cite{kaletha_rigid}.

\sec{Fibers of the map from $H^1(\Gamma,G)\rightarrow H^1(\Gamma,\Gbar)$}
\label{s:fibers}

Suppose $A\subset Z(G)$, and both $G$ and $\Gbar=G/A$ are defined over $\R$.
It is helpful to analyze
the fibers of the map $\psi:H^1(\Gamma,G)\overset\psi\rightarrow H^1(\Gamma,\Gbar)$.
In particular taking $G=\Gsc, \Gbar=\Gad$, 
and summing over $H^1(\Gamma,\Gad)$, we obtain a description 
of $H^1(\Gamma,\Gsc)$, complementary  to that of Proposition \ref{p:H1}.

Suppose $\sigma$ is a real form of $G$ such that $A$ is $\sigma$-invariant.
Write $G(\R,\sigma)=G^\sigma$ and $\Gbar(\R,\sigma)=\Gbar^\sigma$.

Write $p$ for the projection map $G\rightarrow \Gbar$. This restricts to a map
$G(\R,\sigma)\rightarrow \Gbar(\R,\sigma)$,  taking the identity component
of $G(\R,\sigma)$ to that of $\Gbar(\R,\sigma)$. Therefore $p$  factors to a map
(not necessarily an injection):
\begin{subequations}
\renewcommand{\theequation}{\theparentequation)(\alph{equation}}  
\begin{equation}
p^*:\pi_0(G(\R,\sigma))\rightarrow \pi_0(\Gbar(\R,\sigma)).
\end{equation}
Define
\begin{equation}
\pi_0(G,\Gbar,\sigma)=\pi_0(\Gbar(\R,\sigma))/p^*(\pi_0(G(\R,\sigma))).
\end{equation}

There is a natural action of  $\overline G(\R,\sigma)$ on  $H^1(\Gamma,A)$ defined as follows.
Suppose $g\in \Gbar(\R,\sigma)$. Choose $h\in G(\C)$ satisfying $p(h)=g$.
Then $g:a\rightarrow ha\sigma(h\inv)$ factors to a well defined action of 
$\Gbar(\R,\sigma)$ on $H^1(\Gamma,A)$. Furthermore
the image of $G(\R,\sigma)$, which includes the identity component,
acts trivially, so this factors to an  action of $\pi_0(G,\Gbar,\sigma)$. 
\end{subequations}

\begin{proposition}
\label{p:fiber}
Suppose $\gamma\in H^1(\Gamma,G)$, and write $\gamma=[g]$ 
($g\in G^{-\sigma}$). Let $\sigma_\gamma=\int(g)\circ\sigma\in \oInv(G)$.
Then there is a bijection
$$
H^1(\Gamma,G)\supset \psi\inv(\psi(\gamma))\longleftrightarrow H^1(\Gamma,A)/\pi_0(G,\Gbar,\sigma_\gamma).
$$

\end{proposition}

\begin{proof}
First assume $\gamma$ is trivial, and take $g=1$. 
Consider the exact sequence
$$
H^0(\Gamma,G)\rightarrow H^0(\Gamma,\Gbar)\rightarrow H^1(\Gamma,A)\overset{\phi}\rightarrow H^1(\Gamma,G)\overset\psi\rightarrow H^1(\Gamma,\Gbar).
$$
This says $\psi\inv(\psi((\gamma))=\phi(H^1(\Gamma,A))$, i.e. the orbit of the group $H^1(\Gamma,A)$ acting on 
the identity coset.
This  is $H^1(\Gamma,A)$, modulo the action of $H^0(\Gamma,\Gbar)$,
and this action factors through the image of $H^0(\Gamma,G)$. 
The general case follows from an easy twisting argument.
\end{proof}

We specialize to the case $G=\Gsc$ is simply connected and $\overline G=\Gad=\Gsc/\Zsc$ is the adjoint group.

\begin{corollary}
\label{c:fiber} 
Suppose $\sigma$ is a real form of $\Gsc$ and consider the map $\psi:H^1(\Gamma,\Gsc)\rightarrow H^1(\Gamma,\Gad)$. 

Suppose $\gamma\in H^1(\Gamma,\Gad)$, and write $\gamma=[g]$ ($g\in\Gad^{-\sigma}$). 
Let $\sigma_\gamma=\int(g)\circ\sigma$, viewed as an element of $\oInv(\Gsc)$. Then 
\begin{subequations}
\renewcommand{\theequation}{\theparentequation)(\alph{equation}}  
\begin{equation}
\gamma\text{ is in the image of }\psi\Leftrightarrow \zinv(\sigma_\gamma)=\zinv(\sigma),
\end{equation}
in which case
\begin{equation}
|\psi\inv(\gamma)|=|H^1(\Gamma,\Zsc)|/|\pi_0(\Gad(\R,\sigma_\gamma))|.
\end{equation}
Furthermore
\begin{equation}
|H^1(\Gamma,\Gsc)|=|H^1(\Gamma,\Zsc)|\sum_{\substack{\gamma\in H^1(\Gamma,\Gad)\\ \zinv(\sigma_\gamma)=\zinv(\sigma)}}|\pi_0(\Gad(\R),\sigma_\gamma)|\inv
\end{equation}
\end{subequations}
\end{corollary}

\begin{proof}
Statements (b) and (c) follow from the Proposition. 
For (a), when $\sigma=\sigmaqc$ and $\gamma=1$ the proof is immediate, and the general case follows by twisting. 
We leave the details to the reader.
\end{proof}

\sec{Tables}
\label{s:tables}

Most of these results can be computed by hand from Theorem \ref{t:H1}, 
or using Proposition \ref{p:fiber} and the classification of real forms (i.e. the adjoint case). 

The Atlas of Lie Groups and Representations software can be used to calculate
$H^1(\Gamma,G)$ for any real form of a reductive group. This was used to check the tables, and 
for the Spin groups. 
See {\tt www.liegroups.org/tables/galois} for more information about using the software
to compute Galois cohomology.

\subsec{Classical groups}
\label{s:classical}

\begin{tabular}{|l|l|l|}
\hline
Group&$|H^1(\Gamma,G)|$& \\
\hline
$SL(n,\R),GL(n,\R)$ & $1$ &\\  
\hline
$SU(p,q)$ & $ \lfloor\frac p2\rfloor+\lfloor\frac q2\rfloor+1$
&$\begin{gathered}
\text{Hermitian forms of rank $p+q$ and }\\\text{discriminant $(-1)^q$}\end{gathered}$ \\\hline
$SL(n,\mathbb H)$ & $2$ & $\R^*/\text{Nrd}_{\H/\R}(\H^*)$ \\\hline
$Sp(2n,\R)$ & $1$ & real symplectic forms of rank $2n$\\\hline
$Sp(p,q)$ & $p+q+1$ & quaternionic Hermitian forms of rank $p+q$\\\hline
$SO(p,q)$ & 
$\lfloor \frac p2\rfloor+\lfloor \frac q2\rfloor+1$
&
$
\begin{gathered}
\text{real symmetric bilinear forms of rank $n$}\\\text{and discriminant $(-1)^q$}
\end{gathered}
$
\\\hline
$SO^*(2n)$ & 2 &\\
\hline
\end{tabular}

\bigskip

Here $\H$ is the quaternions, and $Nrd$ is the reduced norm map from
$\H^*$ to $\R^*$ (see \cite[Lemma 2.9]{platonov_rapinchuk}).
For more information on Galois cohomology of classical groups 
see
\cite{serre_galois}, \cite[Sections 2.3 and 6.6]{platonov_rapinchuk} and \cite[Chapter VII]{bookofinvolutions}.

\subsec{Simply connected groups}
\label{s:simply}

The only  simply connected  groups with classical root system, which are not  
in the table in Section \ref{s:classical} 
are $Spin(p,q)$ and $Spin^*(2n)$.

Define $\delta(p,q)$ by the following table, depending on $p,q\pmod 4$.

$$
\begin{tabular}{c|cccc}
&0&1&2&3\\\hline
0&3&2&2&2\\
1&2&1&1&0\\
2&2&1&0&0\\
3&2&0&0&0  
\end{tabular}
$$

See Example \ref{ex:spin} for an explanation of these numbers.
\bigskip

\centerline{\begin{tabular}{|l|l|}
\hline
Group&$|H^1(\Gamma,G)|$\\\hline
$Spin(p,q)$& $\lfloor\frac{p+q}4\rfloor +\delta(p,q)$\\  
\hline
$Spin^*(2n)$& $2$\\\hline
\end{tabular}}

\bigskip
\bigskip

\begin{tabular}{|c|c|c|c|c|c|}
\hline
\multicolumn{6}{|c|}{Simply connected exceptional groups}\\
\hline
inner class&group&$K$&real rank&name&$|H^1(\Gamma,G)|$\\
\hline
compact&$E_6$ &$A_5A_1$ & $4$ & $\begin{gathered}\text{quasisplit'}\\ {\text{quaternionic}}\end{gathered} $& 3\\\hline
&$E_6$ &$D_5T$ & $2$ & Hermitian & 3\\\hline
&$E_6$ &$E_6$ & $0$ & compact & 3\\\hline
split&$E_6$ &$C_4$ & $6$ & split & 2\\\hline
&$E_6$ &$F_4$ & $2$ & quasicompact & 2\\\hline\hline
compact& $E_7$ & $A_7$ & $7$ &split & $2 $\\\hline
& $E_7$ & $D_6A_1$ & $4$ &quaternionic & $4$\\\hline
& $E_7$ & $E_6T$ & $3$ &Hermitian & $2$\\\hline
& $E_7$ & $E_7$ & $0$ &compact & $4$\\\hline
\hline
compact& $E_8$ & $D_8$ & $8$ &split & $3$\\\hline
& $E_8$ & $E_7A_1$ & $4$ &quaternionic & $3$\\\hline
& $E_8$ & $E_8$ & $0$ &compact & $3$\\\hline
\hline
compact& $F_4$ & $C_3A_1$ & $4$ &split & $3$\\\hline
& $F_4$ & $B_4$ & $1$ & & $3$\\\hline
& $F_4$ & $F_4$ & $0$ &compact & $3$\\\hline
\hline
compact& $G_2$ & $A_1A_1$ & $2$ &split & $2$\\\hline
& $G_2$ & $G_2$ & $0$ & compact & $2$\\\hline
\end{tabular}

\newpage

\subsec{Adjoint groups}
\label{s:adjoint}

If $G$ is adjoint $|H^1(\Gamma,G)|$ is the number of real forms 
in the given inner class, which is well known. 
We also include the  component group, which is useful in connection with 
Corollary  \ref{c:fiber}.

One technical point arises in the case  of $PSO^*(2n)$. If $n$ is even there are 
two real forms which are related by an outer, but not an inner, automorphism. 
See Remark \ref{r:notserre}.

\bigskip

\begin{tabular}{|c|c|c|}
\hline
\multicolumn{3}{|c|}{Adjoint classical groups}\\
\hline
Group&$|\pi_0(G(\R))|$&$|H^1(\Gamma,G)|$ \\
\hline
$PSL(n,\R)$  &  $
\begin{cases}
  2&n\text{ even}\\
  1&n\text{ odd}\\
\end{cases}$
& $
\begin{cases}
  2&n\text{ even}\\
  1&n\text{ odd}\\
\end{cases}$
\\\hline
$PSL(n,\H)$ &$1$&$2$\\\hline
$PSU(p,q)$ &$
\begin{cases}
  2&p=q\\1&\text{ otherwise}
\end{cases}
$
 & $\lfloor \frac{p+q}2\rfloor+1$\\\hline
$PSO(p,q)$ &
$
\begin{cases}
1&pq=0\\
1&p,q\text{ odd and } p\ne q\\
4&p=q\text{ even}\\
2&\text{otherwise}  
\end{cases}
$
&$
\begin{cases}
\lfloor\frac{p+q+2}4\rfloor&p,q\text{ odd}\\  
\frac{p+q}4+3&p,q\text{ even}, p+q=0\pmod4\\
\frac{p+q-2}4+2&p,q\text{ even}, p+q=2\pmod4\\
\frac{p+q+1}2&p,q\text{ opposite parity}\\  
\end{cases}
$
\\\hline
$PSO^*(2n)$ &
$\begin{cases}
  2&n\text{ even}\\
  1&n\text{ odd}\\
\end{cases}
$
&$
\begin{cases}
\frac n2+3&n\text{ even}  \\
\frac {n-1}2+2&n\text{ odd}  
\end{cases}
$
\\\hline

$PSp(2n,\R)$ &$2$&$\lfloor \frac n2\rfloor +2$\\\hline
$PSp(p,q)$ &
$
\begin{cases}
2&p=q\\
1&else  
\end{cases}
$
&$\lfloor \frac {p+q}2\rfloor +2$\\\hline
\end{tabular}

\bigskip

The groups $E_8,F_4$ and $G_2$ are both simply connected and adjoint. 
Furthermore in type $E_6$ the center of 
the simply connected group $\Gsc$ has order $3$, and 
it follows that $H^1(\Gamma,\Gad)=H^1(\Gamma,\Gsc)$ in these cases.
So the only groups not covered by the table in Section \ref{s:simply} 
are adjoint groups of type $E_7$.

\medskip

\begin{tabular}{|c|c|c|c|c|c|c|}
\hline
\multicolumn{7}{|c|}{Adjoint exceptional groups}\\
\hline
inner class&group&$K$&real rank&name&$\pi_0(G(\R))$ &$|H^1(G)|$\\
\hline
compact& $E_7$ & $A_7$ & $7$ &split & 2&$4$\\\hline
& $E_7$ & $D_6A_1$ & $4$ &quaternionic & $1$ &$4$\\\hline
& $E_7$ & $E_6T$ & $3$ &Hermitian & $2$ &$4$\\\hline
& $E_7$ & $E_7$ & $0$ &compact & $1$&$4$\\\hline
\end{tabular}
 
\bibliographystyle{plain}
\def\cprime{$'$} \def\cftil#1{\ifmmode\setbox7\hbox{$\accent"5E#1$}\else
  \setbox7\hbox{\accent"5E#1}\penalty 10000\relax\fi\raise 1\ht7
  \hbox{\lower1.15ex\hbox to 1\wd7{\hss\accent"7E\hss}}\penalty 10000
  \hskip-1\wd7\penalty 10000\box7}
  \def\cftil#1{\ifmmode\setbox7\hbox{$\accent"5E#1$}\else
  \setbox7\hbox{\accent"5E#1}\penalty 10000\relax\fi\raise 1\ht7
  \hbox{\lower1.15ex\hbox to 1\wd7{\hss\accent"7E\hss}}\penalty 10000
  \hskip-1\wd7\penalty 10000\box7}
  \def\cftil#1{\ifmmode\setbox7\hbox{$\accent"5E#1$}\else
  \setbox7\hbox{\accent"5E#1}\penalty 10000\relax\fi\raise 1\ht7
  \hbox{\lower1.15ex\hbox to 1\wd7{\hss\accent"7E\hss}}\penalty 10000
  \hskip-1\wd7\penalty 10000\box7}
  \def\cftil#1{\ifmmode\setbox7\hbox{$\accent"5E#1$}\else
  \setbox7\hbox{\accent"5E#1}\penalty 10000\relax\fi\raise 1\ht7
  \hbox{\lower1.15ex\hbox to 1\wd7{\hss\accent"7E\hss}}\penalty 10000
  \hskip-1\wd7\penalty 10000\box7} \def\cprime{$'$} \def\cprime{$'$}
  \def\cprime{$'$} \def\cprime{$'$} \def\cprime{$'$} \def\cprime{$'$}
  \def\cprime{$'$} \def\cprime{$'$}

\enddocument
\end